%% file: shell2.tex
\documentclass[12pt,draft]{amsart}

\newtheorem{thm}[subsection]{Theorem}
\newtheorem{lem}[subsection]{Lemma}
\newtheorem{cor}[subsection]{Corollary}
\newtheorem{prop}[subsection]{Proposition}
\newtheorem{obs}[subsection]{Observation}
\newtheorem{conj}[subsection]{Conjecture}

\theoremstyle{definition}
\newtheorem{definition}[subsection]{Definition}
\newtheorem{example}[subsection]{Example}

\newtheorem{remark}[subsection]{Remark}

\newcommand{\R}{\mathbf{R}}
\newcommand{\C}{\mathbf{C}}

\newcommand{\US}{\mathbf{S}}

\numberwithin{equation}{subsection}

\def\a{\alpha}
\def\b{\beta}
\def\g{\gamma}
\def\k{\kappa}
\def\eps{\varepsilon}
\def\i{\mathrm{i}}
\def\w{\omega}
\def\th{\theta}

\def\Lap{\Delta}
\def\can#1{{\Gamma_{#1}}}

\def\On#1{\mathrm{O}(#1)}

\def\sz#1#2{\mathcal{S}_{#1}\!\left(#2\right)}
\def\szd#1#2{\mathcal{S}_{#1}^{*}\!\left(#2\right)}
\def\G#1{\mathrm{G}_{#1}}
\def\E#1#2{\mathcal{E}^{#1}\left(#2\right)}
\def\F{\mathcal{F}}

\def\I{\mathcal{I}}

\def\bw#1{\textstyle{\bigwedge_{#1}}}

\def\^{\wedge}
\def\ip#1#2{\langle#1,#2\rangle}
\def\#{\sharp}
\def\rad{\mathcal{R}}
\def\dar{\mathcal{R}^*}
\def\wtp{\omega_{\lambda,\th}}

\def\knorm#1#2{\left\|#1\right\|_{({#2})}}
\def\tr{\mathrm{tr}}
\def\sv#1{|\US^{#1-1}|}
\def\bn#1#2{\mathrm{Ch}({#1},{#2})}
\def\hom#1{\mathrm{Hom}(#1)}
\def\dit{\,.\,}
\def\bd{\partial}

\catcode`@=11 \@mparswitchfalse  
\newcounter{mnotecount}[page]

\begin{document}

\parskip 5pt
\parindent 0pt
\baselineskip 16pt

\title[X-rays of forms/projections of currents]{X-rays of forms\\and projections of currents}

\author{Bruce Solomon}

\address{Math Department, Indiana University,
Bloomington, IN 47405}
\email{solomon@indiana.edu}
\urladdr{mypage.iu.edu/$\sim$solomon}

\date{This version begun May 12, 2009; Last Typeset \today.}

\begin{abstract}
We study a new Radon-like transform that averages 
projected $p$-forms in $\,\R^{n}\,$ over affine $(n-k)$-spaces. We then prove an 
explicit inversion formula for our transform on the space of rapidly-decaying smooth $p$-forms. Our transform differs from the one in \cite{ggs}. 
Moreover, if it can be extended to a somewhat larger space
of $p$-forms, our inversion formula will allow the synthesis of any rapidly-decaying
smooth $p$-form on $\,\R^{n}\,$ as a (continuous) superposition of pullbacks from 
$p$-forms on $k$-dimensional subspaces. In turn, such synthesis implies an explicit 
formula (which we derive) for reconstructing compactly supported currents in $\,\R^{n}\,$
(e.g., compact oriented $k$-dimensional subvarieties) from their oriented projections onto 
$k$-planes.
\end{abstract}

\maketitle

\section{introduction}\label{sec:intro}
\input{intro2}


\section{Preliminaries}\label{sec:prelims}
\input{prelims2}


\section{The transforms $\rad_{k}\,$ and $\dar_{k}$.}\label{sec:rad}
\input{rad2}


\section{Convolution formula}\label{sec:conv}
\input{conv2}


\section{Kernels}\label{sec:FT}
\input{kernels2}


\section{Inversion, Superposition, and Reconstruction from Projections}\label{sec:inv}
\input{inversion2}


\section{Appendix: Technical lemmas}
\input{tech2}\label{sec:tech}


\input{biblio2}

\vfill

\end{document}

%% file: intro2.tex

\textbf{Question:} 
Do the projections of an oriented $p$-dimensional object in $\,\R^{n}\,$ into $k$-dimensional 
subspaces---cancellation allowed---determine it?  

\textbf{Question:}
Can one use its projections to explicitly reconstruct the object?

Here, these questions prompt us to define a new integral transform of 
``Radon'' type for differential $p$-forms, and to derive a formula that inverts it on the 
space of rapidly-decaying $p$-forms. Our inversion formula doesn't quite settle the questions
above, but it comes very close: We believe it shows exactly what their answers have to be.

By no means are we first to construct a tomography of $p$-forms: 
In the late 1960's, Gelfand, Graev and Shapiro did so in \cite{ggs}, 
and their work was followed by a series of related papers (e.g., \cite{ggg}). 
But our theory complements theirs, and unlike it, bears directly on the questions above.

To understand our transform, and see how it differs from the one in \cite{ggs}, 
recall that any a linear $k$-plane $\,P\,$,
together with its orthogonal complement $\,P'\,$, splits $\,\R^{n}\,$ as the direct sum $\,\R^{n}=P\oplus P'\,$. 
The exterior algebras $\,\bw{*}{P}\,$ and $\,\bw{*}{P'}\,$ then both form subspaces of the larger exterior
algebra $\,\bw{*}{\R^{n}}\,$. Note, however, that their sum does \emph{not} generally {span} 
$\,\bw{*}{\R^{n}}\,$.

Now consider a continuous $p$-form $\,\a\,$ in $\,\R^{n}\,$, and suppose it decays 
fast enough to make it integrable over any affine $\,(n-k)$-plane. Our transform assigns, 
to each $\,\xi\in P\,$, a value $\,\rad_{k}\a(P,\xi)\in\bw{p}{P}\,$, by projecting $\,\a(x)\,$ into 
$\,\bw{p}{P}\,$ at each $\,x\in P'\,$, and then integrating along $\,P'\,$. 
We emphasize that $\,\rad_{k}\a(P,\xi)\,$ thus belongs to $\,\bw{p}{P}\,$.

Like ours, the transform in \cite{ggs} integrates along $\,P'\,$.
But it does so after projecting $\,\a(x)\,$ into $\,\bw{p}{P'}\,$---not
$\,\bw{p}{P}\,$--- at each $\,x\in P'\,$. Accordingly, their transform assigns, to each $\,\xi\in P\,$,
a value in $\,\bw{p}{P'}\,$ instead of $\,\bw{p}{P}\,$. 
Though procedurally similar to ours, this process encodes very different information 
about the input $p$-form $\,\a\,$.

For instance, when $\,\a\,$ is a 1-form in $\,\R^{3}\,$, our 2-plane transform
(which integrates along lines perpendicular to each linear 2-plane) is 
\emph{invertible}, in the sense that we can recover any rapidly decaying 1-form $\,\a\,$ 
from its transform $\,\rad_{2}\a\,$. The transform in \cite{ggs}, contrastingly, is \emph{not} invertible 
in that case. For it annihilates all exact 1-forms. Indeed, its integration step 
computes standard line integrals, and the line integral of a rapidly decaying exact form must
vanish.

Further, the low dimensions of the example just given oversimplify the distinction between
our transform and that of \cite{ggs}. For when $\,k=n-1=2\,$, and $\,p=1\,$, the inclusion 
\begin{equation*}
	\bw{p}{P}\oplus\bw{p}{P'}\subset\bw{p}{\R^{n}}
\end{equation*}
becomes an isomorphism. In that case, adding our transform to the one in \cite{ggs} 
gives the classical (trivially vectorized) Radon transform. Roughly speaking, 
that puts any one of these transforms in the span of the other two,
in which case our work produces nothing fundamentally new. 
But in general---in fact, whenever $\,k>p>1\,$---the inclusion above is proper, 
and this makes our transform truly inequivalent to that of \cite{ggs}.

To explain why our transform and inversion formula are not just new, but also interesting,
we revisit our opening questions. The oriented $p$-dimensional objects we have in mind
there are compactly supported \emph{$p$-dimensional currents}, in the sense of deRham. 
As we explain in \S6, they belong to the dual of the space of smooth differential 
$p$-forms, and they project onto $k$-dimensional subspaces in a simple, natural 
way. Their relationship to $p$-forms makes the problem of reconstructing a current from its projections
equivalent to that of writing any rapidly decaying  $p$-form as a \emph{superposition} of forms that
are \emph{$k$-planar}, that is, pullbacks of $p$-forms from $k$-dimensional subspaces via orthogonal
projection. Formally, one gets precisely the right type of superposition formula from our inversion result,
and when $\,p=0\,$ (where our whole theory reduces to that of the classical scalar Radon 
transform) this formula is known to hold. We feel sure that it holds for $\,p>0\,$ too,
and state the expected facts, along with the answers they give to our opening questions, as explicit
conjectures after proving our inversion result (Theorem \ref{thm:inversion}).
We hope that our work tempts someone having more skill with the required analytical tools 
to tie up these loose ends.

Meanwhile, discovering the appropriate tools, deriving the correct inversion formula, and proving its 
validity on the smaller space of rapidly decaying $p$-forms will occupy us here. 
Before beginning, we briefly preview our plan of attack.

Section \ref{sec:prelims} sets preliminary concepts and notation dealing mainly with exterior algebra and the
Fourier transform. The material is routine, except possibly for Propositions \ref{prop:ftd} and 
\ref{prop:symbols2}, which show how the Fourier transform interwines
 $\,d\,$, $\,\delta\,$, $\,d\delta\,$, and $\,\delta d\,$ with certain operations in the exterior algebra. 
These elegant formulae are easy, but we have not seen them elsewhere, and they
combine to intertwine the half-laplacians $\,d\delta\,$ and $\,\delta d\,$ with the projections we 
call $\,\Pi\,$ and $\,\Phi\,$ in  $\,\bw{*}{\R^{n}}\,$. We later exploit these facts to derive Corollary 
\ref{cor:even}, a simpler ``differential'' statement of our inversion formula valid for even
codimension.

In \S\ref{sec:rad}, we define our transform $\rad_{k}\,$ and its dual $\,\dar_{k}\,$.
We work out a couple examples along the way to clarify definitions 
and suggest the typical behavior of these operators.

In \S \ref{sec:conv}, we start the core work of our paper by deriving an explicit formula 
for the composition $\,\dar_{k}\circ\rad_{k}\,$ as a convolution
(Theorem \ref{prop:conForm}). This type of result---and its utility for inverting a Radon 
transform---were first noted by B.~Fuglede in \cite{fuglede}. Indeed, our formula reduces to Fuglede's result 
when the $p$-forms on which it operates are mere scalar functions ($p=0$). But the precise generalization 
for $\,p>0\,$ requires a more sophisticated analysis focused on the action of the orthogonal group $\,\On{n}\,$
on the exterior algebra $\,\bw{*}{\R^{n}}$. The resulting convolution formula thus forms the first 
significant contribution of our paper.

The other really new element is our computation of the precise Fourier transform 
of $\,\dar_{k}\circ\rad_{k}\,$, obtained as Theorem \ref{thm:fpi} in \S \ref{sec:FT}. 
Roughly speaking, we carry it out by expanding the convolution kernel
from \S \ref{sec:conv} as a sum of quotients of spherical harmonics by powers of the radial 
distance  $\,r\,$. We then use a well-known lemma from Stein \cite{stein} to compute the Fourier 
transforms of these quotients. 

We conclude in \S \ref{sec:inv} by combining the  convolution formula and our Fourier transform 
computation to invert $\,\rad_{k}\,$ (Theorem \ref{thm:inversion}). 
When the codimension $\,n-k\,$ is even,
our inversion formula simplifies: The intertwining rules at the end
of \S\ref{sec:prelims} replace the Fourier transforms
and projections there to simple linear combinations of the ``half-laplacians'' $\,d\delta\,$ and 
$\,\delta d\,$. Corollary \ref{cor:even} states the precise result.

Finally, our Inversion Theorem leads quickly to the formal superposition result mentioned earlier:
We state it as Conjecture \ref{cor:super}, and then show that it would 
answer our opening questions. We present those answers too, as Conjecture \ref{thm:proj}.

We completed this project with sabbatical support from Indiana University, a Lady Davis fellowship at the Technion in Haifa, Israel, and extended hospitality from Stanford University. We deeply appreciate the assistance of all three institutions.


%% file: prelims2.tex

For later reference, we record basic notation and background here, mostly about the 
exterior algebra, the Fourier transform, and their interaction. The reader may prefer  
skipping to \S\ref{sec:rad}, and return here for clarification when the need arises. 

\subsection{Subspaces and complements.} 
Let $\,\G{n,k}\,$  denote the compact Grassmann manifold of all 
$k$-planes in $\,\R^{n}\,$.
We always denote the orthogonal complement of a linear $k$-plane $\,P\in\G{n,k}\,$ by $\,P'\,$, 
and we write $\,k'\,$  for $\,\dim P' = \mathrm{codim}\,P=n-k\,$. 

\subsection{Exterior algebra.}\label{ssec:bw} 
We write $\,\bw{*}V\,$ for the \textbf{exterior algebra} 
of a  finite-dimensional vector space $\,V\,$. Chapter 1 of
Federer's book \cite{federer} gives a careful (albeit terse) 
development of the exterior algebra from scratch. 
Here we review just a few relevant points.

First, $\,\bw{*}V\,$ has the direct sum decomposition
\begin{equation*}\label{eqn:bwds}
	\bw{*}V = \bigoplus_{k=0}^{n}\bw{k}V\ ,
\end{equation*}
where $\,\bw{k}V\,$ denotes the subspace generated by $k$-fold wedge products of the form 
$\,v_{1}\wedge v_{2}\wedge\cdots\wedge v_{k}\,$, with $\,v_{i}\in V\,$.
We call the elements of $\,\bw{k}V\,$ \textbf{$k$-vectors}.

Given an orthonormal basis $\,e_{1}, e_{2},\dots, e_{n}\,$ for
$\,V\,$, we get a basis for $\,\bw{k}V\,$  comprising the 
$\,\binom{n}{k}\,$ $k$-vectors
\begin{equation}\label{eqn:basis}
	e_{\lambda}:=
	e_{\lambda_{1}}\wedge e_{\lambda_{2}}\wedge\cdots e_{\lambda_{k}}\ ,
\end{equation}
where $\,\lambda\,$ runs over all possible {increasing}, length-$k$ subsequences of 
$\,\left\{1,2,\cdots,n\right\}\,$. Declaring this basis orthonormal extends the inner 
product on $\,V\,$ to an inner product $\,\ip{\cdot}{\cdot}\,$ on 
$\,\bw{*}V\,$

We get a linear transformation on $\,\bw{*}V\,$ by left-wedging with any fixed vector $\,v\in V\,$. 
We indicate the adjoint of this mapping (relative to the inner product above on $\,\bw{*}{V}\,$)
by ``$\vee\,v$'' and define it by requiring
\begin{equation}\label{eqn:vee}
	\langle \a\vee v,\ \b\rangle := \langle\a,\ v\wedge\b\rangle
\end{equation}
for all $\,\a,\b\in\bw{*}V\,$.

Any linear map $\,f:V^{}\to W^{}\,$ between vector spaces induces a multilinear extension between 
their respective exterior algebras via the simple formula
\[
v_{1}\^v_{2}\^\cdots\^v_{k}\longmapsto f(v_{1})\^f(v_{2})\^\cdots\^f(v_{k})\ .
\]
Two examples particularly relevant for us arise with respect to a subspace $\,P\subset\R^{n}\,$:

\begin{enumerate} 

	\item
	\textit{Inclusion:}\label{itm:pullback} 
	The inclusion $\,f:P\to\R^{n}\,$ extends to an inclusion 
	$\,\bw{*}P\subset\bw{*}\R^{n}\,$. This lets us regard elements of $\,\bw{*}P\,$ as belonging to 
	$\,\bw{*}\R^{n}\,$. We shall do so below routinely without further comment.  
	\smallskip

	\item\label{itm:restrict}
	\textit{Projection:} The orthogonal projection $\,\R^{n}\to P\,$
	is adjoint to the inclusion $\,P\subset\R^{n}\,$
	with respect to the usual dot product. It extends to a projection 
	$\,\bw{*}\R^{n}\to \bw{*}P\,$ that is orthogonal for 
	the inner product $\,\ip{\cdot}{\cdot}\,$ defined above
	on $\,\bw{*}\R^{n}\,$.  We often suffix a vertical bar to indicate 
	this projection. For instance, we write the projection of a vector 
	$\,v\,$ on $\,\R^{n}\,$ to $\,P\,$ as $\,v\big|_{P}\,$.
\end{enumerate}

Let $\,H_{x}\,$ denote the hyperplane perpendicular to a non-zero $\,x\in\R^{n}\,$. The splitting $\,\R^{n}=H_{x}\oplus H_{x}^{'}\,$ induces a split in the exterior algebra:
\begin{equation}\label{eqn:split1}
	\bw{*}\R^{n}=\bw{*}H_{x}\oplus \left(x\wedge\bw{*}H_{x}\right)\ .
\end{equation}
We designate the projections onto the first and second factors here respectively by
\begin{equation}\label{eqn:piPhi}
	\Pi_{x}:\bw{*}\R^{n}\to \bw{*}H_{x}
	\quad\text{and}\quad
	\Phi_{x}:\bw{*}\R^{n}\to x\wedge\bw{*}H_{x}\ .
\end{equation}
The first of these is induced by the orthogonal
projection $\,\R^{n}\to H_{x}\,$, as in item (\ref{itm:restrict}) above.
We give it a special symbol nonetheless---it plays a key role later on.

Finally, note that $\,\Pi_{x}\,$ acts as the identity---and $\,\Phi_{x}\,$ vanishes---on the 
degree-0 summand $\,\bw{0}{\R^{n}}\,$, because the latter lies entirely in $\,\bw{*}{H_{x}}\,$.

\subsection{Differential forms.}\label{ssec:forms} 
For our purposes, a \textbf{differential $p$-form} on a subspace 
$\,P\subset \R^{n}\,$ is a mapping $\,\phi:P\to\bw{p}\R^{n*}\,$, where $\,\R^{n*}\,$ denotes the 
dual of $\,\R^{n}\,$. The dot product identifies $\,\R^{n}\,$ with $\,\R^{n*}\,$ via
\begin{equation*}
	(p_{1},p_{2},\dots, p_{n})\quad\leftrightarrow\quad
	p_{1}\,dx_{1}+p_{2}\,dx_{2}+\cdots+p_{n}\,dx_{n}\,,
\end{equation*}
and because of this isomorphism, we will frequently ignore the distinction
between vectors and covectors. We often regard maps from $\,\R^{n}\,$ to $\,\bw{p}\R^{n}\,$ 
as differential forms, for instance.

Call a differential form $\,\phi\,$ on $\,\R^{n}\,$  
\textbf{$k$-planar} if we can write it as the ``pullback'' of a
differential form $\,\psi\,$ native to some $k$-dimensional subspace 
$\,P\subset\R^{n}\,$ via orthogonal projection 
$\,\R^{n}\to P\,$. 
Using the notation introduced above, this simply means  
that for some $\,P\in\G{n,k}\,$ and all $\,x\in\R^{n}\,$, 
$\,\phi\,$ satisfies the following two conditions:
	\begin{enumerate}
		\item 
		$\,\phi(x) = \phi(x|_{P})\,$
		\smallskip
		
		\item
		$\,\phi(x) = \phi(x)\big|_{P}\,$
	\end{enumerate}

For instance, the forms commonly denoted by $\,dx+dy\,$ and $\,x\, dx\wedge dy\,$ 
are both $2$-planar on $\,\R^{3}\,$ with $\,P\,$ defined by $\,z=0\,$. 

Contrastingly, neither $\,z\,dx\wedge dy\,$ nor $\,dx\^ dy\^ dz\,$ 
are 2-planar. The former satisfies condition (2) for the plane $\,z=0\,$, but not (1). 
The latter satisfies (1) but fails condition (2) for all $\,P\in\G{2,3}\,$.

\begin{remark}
	The point of Conjecture \ref{cor:super} is that we may construct many $q$-forms in $\,\R^{n}\,$  
	by averaging $k$-planar $q$-forms over $\,\G{n,k}\,$. 
\end{remark}

Next, recall that we take the \textbf{exterior derivative} of a
differential form $\,\phi\,$ on $\,\R^{n}\,$ using the formula
\begin{equation}\label{eqn:d}
	d\phi = 
	\sum_{i=1}^{n}dx_{i}\wedge\frac{\partial\phi}{\partial x_{i}}\,.
\end{equation}
Somewhat less familiar is the degree-decreasing \textbf{divergence operator} $\,\delta\,$. 
Our notation (cf. \ref{eqn:vee}) lets us to express it thus:
\begin{equation}\label{eqn:div}
	\delta\phi 
	= \sum_{i=1}^{n}\frac{\partial\phi}{\partial x_{i}}\vee dx_{i}\ .
\end{equation}
Note that we omit the minus sign often used to define $\,\delta\,$ so that
\begin{equation}\label{eqn:lap}
	(d\,\delta+\delta\,d)\phi 
	\ =\  
	\sum_{i=1}^{n}\frac{\bd^{2}\phi}{\bd x_{i}^{2}}
	\ =:\ 
	\Delta\,\phi\ ,
\end{equation}
making $\,\Delta\,$ denote the \emph{positive} sum of 
pure second partials.
In \S\ref{ssec:ft} we will see the Fourier transform
relate $\,\delta\,d\,$ and $\,d\,\delta\,$ 
to the projections $\,\Pi\,$ and $\,\Phi\,$ of (\ref{eqn:piPhi}) 
above.

\subsection{Mapping spaces.} 
We will encounter tensor fields on $\,\R^{n}\,$, which, like differential forms,  take values in a normed
vectorspace $\,V\,$. To measure decay rates of these mappings near infinity, we define the seminorm
$\,\knorm{F}{s}\,$ for any map 
$\,F:\R^{n}\to V\,$, and any $\,s\ge 0\,$, via 
\begin{equation}\label{eqn:seminorms}
	\knorm{F}{s}
	:=
	\lim_{r\to\infty}\sup_{|x|>r}\,|x|^{s}\left|F(x)\right|\ .
\end{equation}
Assigning a finite value to $\,\knorm{F}{s}\,$ is slightly more
precise than saying ``$\,F\,$ is $\,O(|x|^{-s})\,$,'' which
merely means $\,\knorm{F}{s}<\infty\,$.
Note that when $\,F\,$ is locally integrable, we ensure summability 
on any \emph{proper} subspace of $\,\R^{n}\,$---though not on 
$\,\R^{n}\,$ itself---by assuming $\,\knorm{F}{n}<\infty\,$.

We will say that $\,F:\R^{n}\to V\,$ \textbf{decays rapidly} if 
$\,\knorm{F}{s}<\infty\,$ for every $\,s\ge 0\,$, and we define 
the space $\,\sz{}{\R^{n}, V}\,$ of \textbf{Schwartz $V$-fields}  
as the set of all smooth $\,F:\R^{n}\to V\,$ with rapidly decaying derivatives of all orders. 
To simplify notation, we will abbreviate
	\begin{itemize}
		\item[] 
		$\sz{}{\R^{n},V,W}:=\sz{}{\R^{n},\mathrm{Hom}(V,W)}\,$,
	\end{itemize}
and when $\,P\subset\R^{n}\,$ is a subspace, 
	\begin{itemize}
		\item[] 
		$\sz{p}{P}:= \sz{}{P,\bw{p}{P}}\,$ (Schwartz $p$-forms)
		
		\item[]  
		$\,\sz{}{P}:=\sz{0}{P} = \sz{}{P,\C}\,$ (Schwartz \emph{functions})
	\end{itemize}

\subsection{Tempered $\,(V,W)$-distributions.}\label{ssec:tmpr} 
When $\,V\,$ 
and $\,W\,$ are inner-product spaces, we write $\,\szd{}{\R^{n},V,W}\,$ for the space of continuous linear functionals $\,\sz{}{\R^{n},V}\to W\,$, and call its constituent
functionals \textbf{tempered $(V,W)$ distributions} or
\textbf{$\,\mathrm{Hom}(V,W)$-valued distributions}, or, when 
context allows, simply \textbf{tempered distributions} 
on $\,\R^{n}\,$. When applying such a distribution $\,\tau\,$ 
to a Schwartz $\,V$-field $\,F\,$, we write $\,\tau[F]\,$, 
with square brackets to emphasize the distributional context.

Certain distributions arise as locally integrable operator fields. 
An \textbf{operator field} is a mapping  $\,T\,$ on $\,\R^{n}\,$ 
that assigns a linear transformation $\,T_{x}\in\hom{V,W}\,$ to 
each $\,x\in\R^{n}\,$. When  $\,T\,$ satisfies 
$\,\knorm{T}{p}<\infty\,$ for some $\,p\ge0\,$, and is also locally integrable, 
it represents a tempered $\,\mathrm{Hom}(V,W)$-valued
distribution (which we still call $\,T\,$) via integration:
\begin{equation}\label{eqn:rbi}
	T[F]:=\int_{\R^{n}}T(x)\dit F(x)\ dx\ ,\quad
	\text{for all $\,F\in\sz{}{\R^{n},V}\,$.}
\end{equation}
Here and henceforth the \textbf{``dot'' operator} in ``$A.X$'' 
instructs us to feed the vector $\,X\in V\,$ to the operator $\,A\,$.
This will save us from over-nesting parentheses. 

Equation (\ref{eqn:rbi}) includes $\,\sz{}{\R^{n},V,W}\,$ into 
$\,\szd{}{\R^{n},V,W}\,$. Since the inner product identifies $\,V\,$
with its dual $\,V^{*}\,$, it also identifies 
$\,\sz{}{\R^{n},V}\,$ with $\,\sz{}{\R^{n},V^{*}}\,$, 
so by (\ref{eqn:rbi}) also includes $\,\sz{}{\R^{n},V}\,$ into
$\,\szd{}{\R^{n},V^{*}}\,$.

As in the scalar setting, one can extend, to all tempered $\,(V,W)$-distri-butions, 
many operators initially defined only on some subset that includes $\,\sz{}{\R^{n},V,W}\,$. The classic example is differentiation: 
When $\,T\in\sz{}{\R^{n},V,W}\,$, integration by parts gives
\begin{equation*}\label{eqn:dT}
\left(\frac{\partial T}{\partial x_{i}}\right)[F] 
= T\left[-\frac{\partial F}{\partial x_{i}}\right]\ .
\end{equation*}
The right side above makes sense for any tempered 
$(V,W)$-distribution $\,T\,$, assuming $\,F\in\sz{}{\R^{n},V}\,$,
so the identity above serves to \emph{define} 
$\,\partial T/\partial x_{i}\,$ for any $\,T\in\szd{}{\R^{n},V,W}\,$.

\subsection{Fourier transform.}\label{ssec:ft}
Using the normalizing conventions of Stein \cite{stein}, 
we define the Fourier transform of any integrable $V$-field $\,F\,$
on $\,\R^{n}\,$ as the $V$-field given by  
\begin{equation}\label{eqn:ft}
	\F(F)(\xi):=\int_{\R^{n}}F(x)\,e^{-2\pi\,\i\,\xi\cdot x}\ dx\ .
\end{equation}
Any basic treatment (e.g. \cite{stz}) will establish 
fundamental properties of the scalar ($\,V=\C\,$) Fourier transform 
which extend trivially to our vector-valued setting. 
These include:

	\begin{enumerate}
	
	\item\label{itm:ftsz}
	$\F\,$ maps $\,\sz{}{\R^{n},V}\,$ into itself.
	\medskip 
	
	\item\label{itm:ftiso}
	$\F\,$ is an $\,L^{2}\,$ isometry, and modulo a sign, 
	inverts itself on $\,\sz{}{\R^{n},V}\,$.
	Specifically,
	\[
		\F\F(F) = F^{-}\ .
	\]
	The superscript ``$-$'' here signals composition with 
	reflection through the origin, as in (\ref{eqn:x-x}).
	\medskip

	\item\label{itm:ftcnv}
	Given $\,T\in\sz{}{\R^n,V,W}\,$ and 
	$\,F\in\sz{}{\R^n,V}\,$, $\F\,$ intertwines the 
	\textbf{convolution} product
	$\,T\star F\,$, given by
	\begin{equation*}
	(T\star F)(y) := \int_{\R^{n}}T(y-x)\dit F(x)\ dx\ .
	\end{equation*}
	with functional operation, in the sense that
	\[
	\F\left(T\star F\right) = \F(T)\,.\,\F(F)
	\]
	The definition of convolution doesn't actually require rapid 
	decay of $\,T\,$ and $\,F\,$, but when both are Schwartz, so is 
	$\,T\star F\,$, which allows one to extend the Fourier 
	convolution rule to a distributional setting. We prove general
	statements of this type in our Appendix 
	(Observations \ref{obs:defcon} and \ref{obs:conv}).
	\medskip
		
	\item
	$\F\,$ is ``self-adjoint'' in the sense that 
	when $\,T\in\sz{}{\R^{n},V,W}\,$ and
	$\,F\in\sz{}{\R^{n},V}\,$, we have
	\[
		\int_{\R^{n}}\F(T)\dit F\ dx = 
		\int_{\R^{n}} T\dit \F(F)\ dx.
	\]
	Note that by using
	distributional notation, we can express this last identity as 
	\begin{equation}\label{eqn:ft*}
		\F(T)\left[F\right] = T\left[\F(F)\right]\ ,
	\end{equation}
	which extends the Fourier transform to all of 
	$\,\szd{}{\R^{n},V,W}\,$.
	\medskip
	
	\item
	Finally, recall the useful way that
	the $\,\F\,$ intertwines differentiation and monomial 
	multiplication:
	\smallskip
	
	\emph{For any $\,F\in\sz{}{\R^n,V}\,$, we have the identities}
	\begin{equation}\label{eqn:ftdm}
		\F\left(\frac{\partial F}{\partial x_j}\right)
		= 
		2\pi\i\,\xi_j\,\F(F)\ ,
		\qquad
		\frac{\partial}{\partial\xi_j}\F(F)
		=
		-2\pi\i\,\F\left(x_j\,F\right) 
		\ .
	\end{equation}
	\end{enumerate}
	\medskip

These well-known identities follow directly from the definition of $\,\F\,$, but they imply 
less familiar rules for intertwining with the exterior derivative and divergence 
operators of (\ref{eqn:d}) and (\ref{eqn:div}):

\goodbreak
\begin{prop}\label{prop:ftd}
For any $\,\phi\in\sz{p}{\R^{n}}\,$, we have
\begin{eqnarray*}
	\F\left(d\phi\right)
	=\phantom{-}\i\,\pi\,d|\xi|^{2}\wedge\F(\phi)&\qquad&
	\F\left(\delta\phi\right)
	=\phantom{-}\i\,\pi\,\F(\phi)\vee d|\xi|^{2}\\
	d\,\F(\phi)
	= -\i\,\pi\,\F\left(d|x|^{2}\wedge\phi\right)&\qquad&
	\delta\,\F(\phi)
	= -\i\,\pi\,\F\left(\phi\vee d|x|^{2}\right)\ .
\end{eqnarray*}
\end{prop}

\begin{proof} Using $\,d|x|^{2}=2\sum_{j}x_{j}\,dx_{j}\,$ and
similarly for $\,d|\xi|^{2}\,$, the stated identities follow
easily from (\ref{eqn:ftdm}) above, the adjoint identity 
(\ref{eqn:vee}) relating $\,\vee\,$ to $\,\wedge\,$, and the 
formulae for $\,d\,$ and $\,\delta\,$ given by (\ref{eqn:d})
and (\ref{eqn:div}).
\end{proof} 

Since $\,d|x|^{2}\,$ is dual to $\,2x\,$, one easily
combines the identities above with (\ref{eqn:vee}) to relate the 
``half-laplacians'' $\,d\,\delta\,$ and $\,\delta\,d\,$
(\ref{eqn:lap}) to the projections $\,\Pi\,$ and $\,\Phi\,$ as
promised at the end of \S\ref{ssec:forms} above:

\begin{prop}\label{prop:symbols2}
Let $\,r:=|x|\,$ denote the radial distance function 
on $\,\R^{n}\,$. Then we have the following identities on 
$\,\sz{p}{\R^{n}}\,$:
\begin{eqnarray*}
	d\,\delta = -4\pi^{2}\,\F^{-1}\circ r^{2}\,\Phi\circ\F\ ,&\quad&
	\delta\,d = -4\pi^{2}\,\F^{-1}\circ r^{2}\,\Pi\circ\F
\end{eqnarray*}
and since $\,d\,\delta+\delta\,d=\Delta\,$, 
\begin{equation*}
	\Delta = -4\pi^{2}\,\F^{-1}\circ r^{2}\circ\F\  .
\end{equation*}
\end{prop}
 
We don't know any reference for the first two identities here, even though the last 
is common knowledge.



%% file: rad2.tex

\subsection{The transform.}

The transforms we describe in this section exchange differential forms on $\,\R^{n}\,$ 
with tensor fields on the \textbf{canonical bundle} $\,\can{n,k}\,$ over $\,\G{n,k}\,$.
We define the latter bundle via
\[
	\can{n,k}
	:=
	\left\{(P,\xi)\in\G{n,k}\times\R^{n}\ :\ \xi\in P\right\}\ .
\]

\begin{definition}[The transform $\,\rad_{k}\,$]
Given a continuous differential $p$-form $\,\a\,$ on $\,\R^{n}\,$ 
with $\,\knorm{\a}{n}<\infty\,$, we denote its \textbf{transform} by
\[
	\rad_{k}\a:\can{n,k}\to\bw{p}\,\R^{n}\ ,
\]
and define it as an integral:
\begin{equation}\label{eq:rad}
	\rad_{k}\a(P,\xi):=\int_{P'}\a(\xi+\eta)\Big|_{P}\ d\eta\ .
\end{equation}
\end{definition}

\begin{remark}\label{rem:rt@p=0}
	When $\,p=0\,$, so that $\,\a\,$ is just a scalar-valued function, we have $\,\a\big|_{P}=\a\,$.
	In this case, $\,\rad_{k}\,$ reduces to the classical $\,k'$-plane transform 
	(cf.~\cite[p.$\,$28]{helgason}).
\end{remark}

\begin{remark}\label{rem:subbdl}
	The definition above characterizes $\,\rad_{k}\a\,$ as a mapping from 
	$\,\can{n,k}\,$ to $\,\bw{*}\R^{n}\,$, making it a section of the trivial $\,\bw{*}\R^{n}\,$ bundle over 
	$\,\can{n,k}\,$. Actually, $\,\rad_{k}\a\,$ takes values in the sub-bundle whose fiber over 
	$\,(P,\xi)\,$ is the pullback subspace $\,\bw{p}P\subset\bw{p}\,\R^{n}\,$. 
\end{remark}

\begin{remark}\label{rem:integration}
	The integral above does \textbf{not} signify oriented integration of forms in the Stokes-theoretic 
	sense. For one thing, the integrand belongs pointwise to $\,\bw{p}P\,$, hence \emph{vanishes} 
	on any $k'$-tuple of vectors tangent to the domain of integration $\,P'\,$. 
	Rather, the integral performs elementary vector integration of a $\,\bw{p}P$-valued function 
	against the Lebesgue measure on $\,P'$. The following example illustrates.
\end{remark}

\begin{example}
Suppose $\, a \in\bw{p}\R^{n}\,$, let $\,\lambda>0\,$, and define 
the $p$-form $\,\phi(x):=e^{-\lambda |x|^{2}} a \,$. 
We compute $\,\rad_{k}\phi\,$ using the well-known identity
\begin{equation}\label{eqn:gauss int}
	\int_{\R^n} e^{-\lambda\,|x|^2}\ dx 
	= 
	\left(\frac{\pi}{\lambda}\right)^{n/2}\ ,
\end{equation} 
which implies, for any $\,(P,\xi)\in\can{n,k}\,$, that
\begin{equation*}
	\int_{P'}e^{-\lambda\left|\xi+\eta\right|^{2}}\ d\eta
	= 
	\int_{P'}e^{-\lambda|\xi|^{2}}e^{-\lambda|\eta|^{2}}\ d\eta
	= 
	\left(\frac{\pi}{\lambda}\right)^{k'/2}e^{-\lambda|\xi|^{2}}\,,
\end{equation*}
whence
\begin{equation*}\label{eq:radGauss}
	\rad_{k}\phi(P,\xi) 
	= 
	\int_{P'}e^{-\lambda|\xi+\eta|^{2}} a \big|_{P}\ d\eta
	= 
	\left(\frac{\pi}{ \lambda}\right)^{k'/2}\,
	e^{-\lambda|\xi|^{2}}a\big|_	{P}\ .
\end{equation*}
\end{example}

The input $p$-form $\,\phi\,$ above decays rapidly, and on each fiber 
of $\,\can{n,k}\,$, its transform $\,\rad_{k}\phi\,$ has the same 
property. The proposition below shows that this is no accident.

Let $\,|\US^{n}|\,$ denote the $n$-dimensional volume of the 
unit sphere in $\,\R^{n+1}\,$.

\goodbreak
\begin{prop}\label{prop:Rdecay}
Suppose $\,\a\,$ is $p$-form on $\,\R^{n}\,$, and $\,s\ge n\,$. 
Then for any $\,P\in\G{n,k}\,$, we have
\[
	\knorm{\rad_{k}\a(P,\,\cdot\,)}{s -k'}
	< 
	4 |\US^{k'-1}|\,\knorm{\a}{s}
\]
\end{prop}

\begin{proof}
Without loss of generality, assume $\,\knorm{\a}{s}<\infty\,$ for 
some $\,s\ge n\,$. Then for sufficiently large $\,\xi\,$ and any 
$\,\eta\in P'\,$, we have 
\[
	\left|\a(\xi+\eta)\right|
	<
	2\knorm{\a}{s}\left(1+|\xi|^{2}+|\eta|^{2}\right)^{-s/2}\ ,
\]
hence also
\begin{eqnarray*}
	\left|\rad_{k}\a(P,\xi)\right| 
	&\le& 
	2\knorm{\a}{s} \int_{P'}\left(1+|\xi|^{2}+|\eta|^{2}\right)^{-s/2}
	\ d\eta\\
	&\le& 
	2\knorm{\a}{s}(1+|\xi|^{2})^{-s/2} \int_{P'}
	\left(1+\frac{|\eta|^{2}}{ 1+|\xi|^{2}}\right)^{-s/2}\ d\eta\\
\end{eqnarray*}
Switching to polar coordinates and changing the radial variable 
$\,r=|\eta|\,$ to $\,\rho=r/\sqrt{1+|\xi|^{2}}\,$, we now get
\begin{eqnarray*}
	\lefteqn{\left|\rad_{k}\a(P,\xi)\right|}&&\\
	&\le& 
	2\knorm{\a}{s}|\US^{k'-1}|(1+|\xi|^{2})^\frac{k'-s}{ 2} 
	\int_{0}^{\infty}\left(1+\rho^{2}\right)^{-s/2}\rho^{k'-1} 
	d\rho\ .
\end{eqnarray*}
But
\begin{eqnarray*}
	\int_{0}^{\infty}\left(1+\rho^{2}\right)^{-s/2}\rho^{k'-1} 
	d\rho
	&<& 
	1+ \int_{1}^{\infty}\left(1+\rho^{2}\right)^{-s/2}\rho^{k'-1} 
	d\rho \\
	&<& 
	1+ \int_{1}^{\infty}\rho^{-s+k'-1} d\rho\ ,
\end{eqnarray*}
and  $\,-s+k'-1\le -n+k'-1=-k-1< -1\,$, because $\,s\ge n\,$. 
The last integral above is therefore bounded by $\,1/k\le1\,$.
\end{proof}

\subsection{The dual transform $\,\dar\,$.} 

As in the classical Radon theory, our transform $\,\rad_{k}\,$ has a formal $\,L^{2}\,$ dual. 
In our situation, the dual sends continuous $\,\bw{p}\R^{n}\,$-valued fields on  $\,\can{n,k}\,$ 
back to $p$-forms on $\,\R^{n}\,$. 

Recall that the transitive action of $\,\On{n}\,$ determines a unique Haar probability measure on 
$\,\G{n,k}\,$. We compute all integrals over $\,\G{n,k}\,$ using that measure.
 
\begin{definition}\label{def:dual}
	The \textbf{dual transform $\,\dar_{k}\,$} takes a continuous 
	mapping $\,\beta:\can{n,k}\to\bw{p}\R^{n}\,$ to the $p$-form 
	on $\,\R^{n}\,$ given by
	\begin{equation*}
		\dar_{k}\beta(x) 
		:= 
		\int_{\G{n,k}} \beta(P,Px)\Big|_{P}\ dP\ ,
	\end{equation*}
	Here $\,Px\,$ denotes the orthogonal projection of $\,x\,$ 
	onto $\,P\,$.
	
	Note that we always have
	\begin{equation}\label{eqn:subbundle}
		\dar_{k}\b = \dar_{k}\bar\b
		\quad\text{when}\quad 
		\left(\bar\b(P,\cdot)-\b(P,\cdot)\right)\Big|_{P}\equiv 0\ .
	\end{equation}
	For each $\,x\in\R^{n}\,$, the map $\,P\mapsto (P,Px)\,$ embeds 
	$\,\G{n,k}\,$ into $\,\can{n,k}\,$. The compactness of 
	$\,\G{n,k}\,$ and continuity of $\,\b\,$ therefore guarantee
	 existence of the integral defining $\,\dar_{k}\,$.
\end{definition}

\begin{remark}\label{rem:superpose}
	Note too that for each fixed $\,P\in\G{n,k}\,$, the integrand 
	$\,\beta(P,Px)\big|_{P}\,$ specifies a \emph{$k$-planar}  $p$-form on 
	$\,\R^{n}\,$. It follows that anything in the image of  $\,\dar_{k}\,$ is a 
	\emph{superposition, over all $\,P\in\G{n,k}\,$}, of $k$-planar 
	$p$-forms. The statement of Conjecture \ref{cor:super} depends on this fact.
\end{remark}

\begin{example}
	Fixing an arbitary $\,v\in\R^{3}\,$, let us compute the dual transform $\,\dar_{2}\b\,$ of 
	the simple field $\,\b:\can{3,2}\to\bw{1}\R^{3}\,$ given by
	\[
		\b(P,\xi):= e^{-|\xi|^{2}}v\,.
	\]
	By definition,
	\[
		\dar_{2}\b(x)
		=
		\int_{\G{2,3}}\b(P,Px)\Big|_{P}\ dP
		=
		\frac{1}{ 4\pi}
		\int_{\US^{2}}e^{-|\w^{\perp}x|^{2}}v\Big|_{\w^{\perp}}\ d\w\ .
	\]
	Here $\,\w^{\perp}\,$ denotes the plane orthogonal to $\,\w\in\US^{2}\,$, and we have 
	used the fact that $\,\US^{2}\,$ double-covers $\,\G{3,2}\,$ homogeneously.
	
	Note also that by virtue of Example \ref{eq:radGauss} and (\ref{eqn:subbundle}), we have
	\begin{equation}\label{eqn:r*r}
		\dar_{2}\b=\dar_{2}\rad_{2}\phi\ ,
	\end{equation}
	where $\,\phi\,$ is the 1-form on $\,\R^{3}\,$ given by $\,\phi(x) = \pi^{-1/2}\,e^{-|x|^{2}}v\,$.
	
	Write $\,r=|x|\,$, $\,u=x/r\,$, parametrize the unit circle perpendicular to $\,u\,$ with 
	unit speed as $\,u^{\perp}(\th)\,$, and put ``$u$-polar'' coordinates on $\,\US^{2}\,$ by setting
	\[
		\w=\wtp:=\lambda\,u + \sqrt{1-\lambda^{2}}\,u_{\th}^{\perp}\ ,
	\]
	Then
	$\,|\w^{\perp}x|^{2}=r^{2}|\w -(\w\cdot u)u|^{2}=
	r^{2}(1-\lambda^{2})\,$, and hence
	\begin{eqnarray*}
		\lefteqn{\dar_{2}\b(x)}\\ 
		&=& 
		\frac{1}{4\pi}
		\int_{-1}^{1}\int_{0}^{2\pi} 
			e^{-r^{2}(1-\lambda^{2})}\,
			\left(v-(v\cdot\wtp)\,\wtp\right)
			\ d\th\,d\lambda\\
		&=& 
		\frac{e^{-r^{2}}}{2\pi}
		\int_{0}^{1} 
			e^{r^{2}\lambda^{2}}\,
			\left(2\pi\,v -
				\int_{0}^{2\pi}
					\left(v\cdot\wtp)\,\wtp\right)\ 
				d\th
			\right)
		\ d\lambda\ ,\\
	\end{eqnarray*}
	thanks to the symmetry of the integrand around 
	$\,\lambda=0\,$.
	
	Now split $\,v\,$ as $\,v=v^{||}+v^{\perp}\,$, with $\,v^{||}\,$ parallel to $\,u\,$, and 
	$\,v^{\perp}\cdot u=0\,$. A routine exercise then finds the parenthesized part of the 
	integrand 	above to be
	\[
		\pi\,\left(1+\lambda^{2}\right)\,v^{\perp}
		+
		2\pi(1-\lambda^{2})\,v^{||}\ .
	\]
	Put this into the last integral for $\,\dar_{2}\b\,$ above, and make the substitution 
	$\,s:=r\,\lambda\,$ to get
	\[
		\rad_{2}\b(x)  
		=
		\textstyle{\frac{1}{2}}\,
		\I_{+}(r)\,v^{\perp} + \I_{-}(r)\,v^{||}\ ,
	\]
	where
	\begin{eqnarray*}
		\I_{\pm}(r)   
		&=&
		r^{-3}\,e^{-r^{2}}\,
		\int_{0}^{r}\,e^{s^{2}}\,(r^{2}\pm s^{2})\ ds\ .
	\end{eqnarray*}
	Both $\,\I_{+}\,$ and $\,\I_{-}\,$ are smooth and decrease
	to zero as $\,r\to\infty\,$. Specifically,
	it is not hard to show, using L'Hospital's rule, that
	\begin{eqnarray*}
		\lim_{r\to\infty}r^{2}\,\I_{+}(r) &=& 1\\
		\lim_{r\to\infty}r^{4}\,\I_{-}(r) &=& 1/2\ ,
	\end{eqnarray*}
	%
Putting these facts together with (\ref{eqn:r*r}), we now see that
\[
	\dar_{2}\b=\dar_{2}\rad_{2}\phi
	\sim 
	\pi\left(\frac{v^{\perp}}{r^{2}}+\frac{v^{||}}{r^{4}}\right)
	\quad\text{as $\,r\to\infty\,$}
\] 
In particular, we get decay, \emph{but not rapid decay, 
despite the rapid decay of the input field $\,\b\,$.} 
This is typical, as will follow from the Convolution formula 
(Prop.~\ref{prop:conForm}), and in this regard, the dual transform 
$\,\dar_{k}\,$ behaves quite differently than the forward transform $\,\rad_{k}\,$.

\end{example}

%% file: conv2.tex

We now want to show that the composition $\,\dar_{k}\circ\rad_{k}\,$, applied to any locally integrable 
and suitably decaying $p$-form on $\,\R^{n}\,$, \emph{convolves} that $p$-form with $\,|x|^{-k}\,\Pi\,$,
up to some dimensional constant. As noted in our introduction, this fact generalizes
the classical result of Fuglede \cite{fuglede}.

Recall (\S\ref{ssec:bw}) that $\,\Pi:\R^{n}\to\mathrm{Hom}(\bw{*}\,\R^{n})\,$ is the operator field we 
get on $\,\R^{n}\,$ by mapping each non-zero $\,x\in\R^{n}\,$ to the orthogonal projection 
$\,\Pi_{x}:\R^{n}\to\bw{*}H_{x}\,$, where $\,H_{x}\,$ is the hyperplane perpendicular to $\,x\,$.
 
                       
The grassmannian $\,\G{n,k}\,$ and the unit sphere $\,\US^{n-1}\subset\R^{n}\,$ are both homogeneous spaces of $\,\On{n}\,$. To rewrite $\,\dar_{k}\circ\rad_{k}\,$ as a convolution, we will first pull integrals over the grassmannian back to integrals over $\,\On{n}\,$, and then push them down to $\,\US^{n-1}\,$. The following lemmas let us
to do so with precision.

\subsection{Two averaging lemmas.} 
The orthogonal group $\,\On{n}\,$ acts on $\,\R^{n}$ (and thus also on $\,\G{n,k}$) by left-multiplication. 
For any $\,a\in\US^{n-1}\,$, denote the $a$-stabilizing subgroup of $\,\On{n}\,$ by 
$\,K_{a}\approx\On{n-1}\,$. Then $\,\On{n}\,$ is foliated by the mutually isometric left cosets of 
$\,K_{a}\,$, and if $\,\g\in\On{n}\,$, with $\,b:=\g a\,$, we define the coset
\begin{equation}\label{eqn:Kpq}
	K_{ab}
	:=
	\g\,K_{a}
	= 
	K_{b}\g
	= 
	\left\{\k\in\On{n}\,:\ \k a = b\right\}\ .
\end{equation}
Note that $\,\On{n}\,$ has a left-invariant metric whose Hausdorff measure assigns total mass 1 to both 
$\,\On{n}\,$ and $\,K_{a}\,$---hence also to each coset $\,K_{ab}\,$---in the appropriate dimensions. 
For integration, we always use these measures.

We have adapted the first formula below from Helgason \cite[I (15)]{helgason}. We need the second one to handle differential forms, as opposed to the scalar-valued functions treated there. 

\begin{lem}\label{lem:avg1} 
Suppose $\,\phi\,$ maps $\,\G{n,k}\,$ continuously into a vectorspace $\,V\,$. Then for any fixed $\,Q\in\G{n,k}\,$, we have
\[
	\int_{\G{n,k}}\phi(P)\ dP 
	= 
	\int_{\On{n}}\phi(\g Q)\ d\g\ .
\]
Further, if $\,a\in\US^{n-1}\,$ and $\,\psi:\On{n}\to V\,$ is continuous, then
\[
	\int_{\On{n}}\psi(\g)\ d\g 
	= 
	\frac{1}{\sv{n}}\,
	\int_{\US^{n-1}}\left(\int_{K_{ab}}\psi(\k)\ d\k\right)\ db\ .
\]
Here $\,db\,$ and $\,d\k\,$ denote the Hausdorff measures on 
$\,\US^{n-1}\,$ and $\,K_{a}\,$ respectively.
\end{lem}

\begin{proof} 
Choose an arbitary $k$-plane $\,P\in\G{n,k}\,$, and 
map $\,\On{n}\stackrel{e_{P}}{\longrightarrow}\G{n,k}\,$ by defining 
$\,e_{P}(\g):=\g\,P\,$. Then $\,e_{P}(g\g)=g\,e_{P}(\g)\,$ for all 
$\,g,\g\in\On{n}\,$, and hence $\,e_{P}\,$ pushes the invariant probability measure on $\,\On{n}\,$ down to the one on $\,\G{n,k}\,$. This lets us integrate $\,\phi\,$ over $\,\G{n,k}\,$ relative to the latter measure by integrating $\,\phi\,\circ\, e_{P}\,$ relative to the former. The first formula asserts nothing more than that.

To get the second formula, we similarly map 
$\,\On{n}\stackrel{e_{a}}{\longrightarrow}\US^{n-1}\,$ via 
$\,e_{a}(\g):=\g\,a\,$.  
Then for any $\,b\in\US^{n-1}\,$, we have 
$\,e_{a}^{-1}(b)=K_{ab}\,$, and the coarea formula \cite[3.2.22]{federer} yields
\begin{equation*}
	\int_{\On{n}}\psi(\g)\,Je_{a}\ d\g 
	= 
	\int_{\US^{n-1}}\left(\int_{K_{ab}}\psi(\k)\ d\k\right)\ d\g\ , 
\end{equation*}
with $\,Je_{a}\,$ denoting the Jacobian of $\,e_{a}\,$. But this Jacobian must be constant, because 
$\,e_{a}\,$ commutes with the transitive left $\,\On{n}\,$ action on itself. Set
$\,\psi\equiv 1\,$ to see that $\,Je_{a}=\sv{n}\,$.
\end{proof}

\textit{Notation:} 
We write $\,\bn{n}{m}\,$ as a horizontal alternative to the usual
vertical symbol for binomial coefficients:
\[
	\bn{n}{m}:=\binom{n}{m}= \frac{n!}{ m!(n-m)!}\ .
\]

\begin{lem}\label{lem:avg2}
Suppose $\,P\in\G{n,k}\,$, $\,x\in P'\,$ and $\,H_{x}\,$ is the hyperplane perpendicular to $\,x\,$ (cf.  \S\ref{ssec:bw}). Then for any $\,\a\in\bw{p}\,\R^{n}\,$ with  $\,p\le k\,$,  we have
\begin{equation*}\label{eqn:avg2}
	\int_{K_{x}}\a\big|_{\k P}\ d\k 
	= 
	\frac{\bn{k}{p}}{ \bn{n-1}{p}}\ \a\big|_{H_{x}}\ .
\end{equation*}
\end{lem}

\begin{proof} Let $\,A\,$ denote the ``averaging'' operator $\,\a\mapsto \int_{K_{x}}\a\big|_{\k P}\ d\k\,$ on the left-hand side of the formula above. Restriction to a subspace (followed by inclusion in $\,\bw{p}\,\R^{n}\,$; see \S\ref{ssec:bw} (\ref{itm:pullback}) above) is a symmetric endomorphism of $\,\bw{p}\,\R^{n}\,$, so $\,A\,$ is symmetric too. 

Now observe that the action of $\,K_{x}\,$ on $\,\bw{p}\R^{n}\,$ respects the splitting (\ref{eqn:split1}), namely
\[
	\bw{p}\,\R^{n} 
	= 
	\bw{p}H_{x}\oplus \left(x\wedge\bw{p-1}H_{x}\right)\ .
\]
It follows that $\,A\,$ commutes with this  action. 
But $\,K_{x}\,$ is the full orthogonal group of  $\,H_{x}\,$. 
So it acts transitively on both $p$- and $(p-1)$-dimensional 
subspaces of $\,H_{x}\,$, hence irreducibly on 
both $\,\bw{p}H_{x}\,$ and $\,x\wedge\bw{p-1}H_{x}\,$.  
This makes each summand in the splitting above an eigenspace of 
$\,A\,$, so that in the notation of \S\ref{ssec:bw}, we have
\begin{equation}\label{eqn:avg2.1}
	A 
	= 
	\lambda\,\Pi_{x} + \mu\,\Phi_{x}\ 
\end{equation} 
for some $\,\lambda,\mu\in\R\,$.
But the definition of $\,A\,$ makes clear that $\,\mu=0\,$, which, 
modulo the particular value of $\,\lambda\,$, proves our lemma. 

To evaluate $\,\lambda\,$, observe that for any subspace 
$\,V\subset\R^{n}\,$, the projection $\,\bw{p}\,\R^{n}\to\bw{p}V\,$ 
has trace given by
\[
	\tr\left(\cdot\,\big|_{V}\right) 
	= 
	\dim\bw{p}V 
	= 
	\left(\!\!\begin{array}{c}\dim V\\p\end{array}\!\!\right)
\]
on $\,\bw{p}\,\R^{n}$. Taking the trace on both sides of (\ref{eqn:avg2.1}) then gives\smallskip
\[
	\tr(A) = \lambda\,\binom{n-1}{p}\ .
\]

At the same time the definition of $\,A\,$ gives\medskip
\[
	\tr(A) 
	= 
	\int_{K_{x}}\tr\left(\cdot\,\big|_{\k P}\right)\ d\k =\binom{k}{p}\ .
\]

These two facts determine $\,\lambda\,$ and prove the lemma.
\end{proof}

      
\goodbreak
We can now present the main result of this section:

\begin{prop}[Convolution formula]\label{prop:conForm}
	Given a locally integrable $p$-form $\,\a\,$ on $\,\R^{n}\,$
	with $\,0\le p\le k < n\,$ and $\,\knorm{\a}{n}<\infty\,$, the 	
	convolution $\,\left(r^{-k}\Pi\right)\star\a\,$ exists, 
	and we have
	\[
		\dar_{k}\,\rad_{k}\a
		= 
		\frac{\sv{k'}\,\bn{k'}{p}}{\sv{n}\,\bn{n-1}{p}}\,
		(r^{-k}\,\Pi)\star\a\ .
	\]
\end{prop}

\begin{proof} 
Our assumptions ensure that for each fixed $\,x\in \R^{n}\,$,
the $p$-form
\[
	y\longmapsto |y|^{-k}\,\Pi_{y}\,\a(y-x)
\]
is locally integrable and $\,O(|y|^{-n-k})\,$ near infinity.
This guarantees existence of the convolution.   

To see that the convolution encodes 
$\,\dar_{k}\,\rad_{k}\,$ 
as claimed, combine the definitions of $\,\rad_{k}\,$ and 
$\,\dar_{k}\,$ with the fact that for any $\,x\in\R^{n}\,$ and 
$\,P\in\G{n,k}\,$, we have $\,Px+P'=x+P'\,$ as affine subspaces.
It follows easily that
\begin{equation*}\label{eqn:darad1}
	\left(\dar_{k}\circ\rad_{k}\right)\a(x) 
	= \int_{\G{n,k}}\int_{P'}\a(x+\eta)\big|_{P}\ d\eta\ dP\ .
\end{equation*}
Now fix an arbitrary $\,Q\in\G{n,k}\,$ and apply the first formula 
in Lemma \ref{lem:avg1}, then Fubini. 
Since $\,\g Q' = \{\g\eta\colon \eta\in Q'\}\,$ 
for any $\,\g\in\On{n}\,$, this gives
\begin{eqnarray}\label{eqn:darad2}
	\left(\dar_{k}\circ\rad_{k}\right)\a(x)
	&=&\int_{Q'}\int_{\On{n}}\a(x+\g\eta)\big|_{\g Q}\  d\g\ d\eta\ .
\end{eqnarray}
Note that when $\,\a\,$ has degree $\,p=0\,$, 
it takes mere scalar values, and the restriction to $\,\g Q\,$ 
has no effect. In that case, 
the inner integral reduces to a simple $x$-centered spherical 
average of $\,\a\,$, and the convolution formula quickly follows
(cf. \cite[Ch.I (34)]{helgason}). Because we assume 
$\,p>0\,$, however, the restriction complicates our task and
we must now sort that out.

Start by using the second identity from Lemma \ref{lem:avg1}. 
If we define $\,\bar\eta=\eta/|\eta|\,$, and recall from 
(\ref{eqn:Kpq}) that $\,K_{\bar\eta q'}= K_{q'}\g\,$ when
$\,\g\in\On{n}\,$ satisfies $\,\g\bar\eta=q'\,$, it rewrites 
(\ref{eqn:darad2}) as
\begin{eqnarray*}
	\dar_{k}\,\rad_{k}\,\a(x)
	&=&
	\frac{1}{\sv{n}}\int_{Q'}\int_{\US^{n-1}}
	\left(\int_{K_{\bar\eta q'}}\a(x+|\eta|q')\big|_{\k Q}\  
	d\k\right)dq'\ d\eta\\
	&=&
	\frac{1}{\sv{n}}\int_{Q'}\int_{\US^{n-1}}
	\left(\int_{K_{q'}}\a(x+|\eta|q')\big|_{\k\g Q}\  
	d\k\right)dq'\ d\eta
\end{eqnarray*}
where we have fixed some $\,\g\in K_{\bar\eta q'}\,$. 
Since $\,\g Q\,$ is a $k$-dimensional subspace of 
$\,T_{q'}\US^{n-1}\,$, Lemma \ref{lem:avg2} now applies to give
\begin{eqnarray*}
	\lefteqn{\dar_{k}\,\rad_{k}\,\a(x)}\\
	&=&
	\frac{1}{\sv{n}}\int_{Q'}\int_{\US^{n-1}}
	\left(\int_{K_{\bar\eta}}\a(x+|\eta|q')\big|_{\k\g Q}\  
	d\k\right)dq'\ d\eta\\ 
	&&\\
	&=&
	\frac{\bn{k'}{p}}{\sv{n}\,\bn{n-1}{p}}
	\int_{Q'}\ \int_{\US^{n-1}}\a(x+|\eta|q')\big|_{H_{q'}}\ 
	dq'\ d\eta\\ 
	&&\\
	&=&
	\frac{\sv{k'}\,\bn{k'}{p}}{\sv{n}\,\bn{n-1}{p}}\ 
	\int_0^{\infty}\int_{\US^{n-1}}\a(x+r\,q')
	\big|_{H_{q'}}\,r^{k'-1}\ dq'\ dr\ ,
\end{eqnarray*}
switching from rectangular to polar coordinates on $\,Q'\,$, with $\,r:=|\eta|\,$. But by writing 
$\,r^{k'-1}=r^{n-1}\,r^{-k}\,$, $\,y:=-r\,q'\,$, and noting that $\,H_{y}=H_{q'}\,$, we can 
make the reverse move on $\,\R^{n}\,$, to deduce
\begin{equation*}\label{eqn:cnv}
	\dar_{k}\,\rad_{k}\,\a(x)
	=
	\frac{\sv{k'}\,\bn{k'}{p}}{\sv{n}\,\bn{n-1}{p}}\ 
	\int_{\R^{n}}\a(x-y)\big|_{H_{y}}\,r^{-k}\ dy\ .
\end{equation*}
Now simply recall from the end of \S\ref{ssec:bw} that 
restriction to $\,H_{y}\,$ induces the projection we call
$\,\Pi_{y}\,$, so that
\[
	\a(x-y)\Big|_{H_{y}}r^{-k} = |y|^{-k}\,{\Pi_{y}}\,\a(x-y)\ .
\]
Putting this into the equation above, we obtain our convolution
formula.

\end{proof}

\begin{remark} 
	Though we don't include the proof here, the convolution above must decay near infinity.
	To be precise, one can show that as $\,r:=|x|\to\infty\,$,	
	\[
		\left|\left({r^{-k}}\,T\right)\star f(x)\right|=
		\begin{cases} 
		O\left(r^{-k}\right)\ ,&  \|f\|_{L^{1}}<\infty\\
		O\left(r^{-k}\ln r\right)\ , & \|f\|_{L^{1}}=\infty
		\end{cases}
	\]
	for any locally integrable, vector-valued mapping $\,f\,$ on $\,\R^{n}\,$, 
	assuming that $\,k<n\,$, $\,\knorm{f}{n}<\infty\,$, and $\,T\,$ is a bounded operator field.
\end{remark}

\begin{remark}\label{rem:cf@p=0}
	As noted in Remark \ref{rem:rt@p=0}, our transform reduces to the classical Radon
	$k'$-plane transform when $\,p=0\,$. The reader will easily check that
	when we apply our convolution formula to any continuous $0$-form 
	$\,f:\R^{n}\to\R\,$ with $\,\knorm{f}{n}<\infty\,$, it too reduces to a known result 
	(cf. \cite[p.$\,$29 (55)]{helgason}), namely
	\[
		\dar_{k}\circ\rad_{k}f= \frac{\sv{k'}}{\sv{n}}\,
		r^{-k}\star f\,.
	\]
\end{remark}

\begin{remark}\label{rem:smooth}
Roughly speaking, $\,\dar_{k}\rad_{k}\a\,$ will be at least as smooth as $\,\a\,$ itself. 
For when$\,\a\,$ has continuous derivatives $\,D^{\g}\a\,$ for all multi-indices 
$\,|\g|\le q\,$, with $\,|D^{\g}\a| = O(r^{-n})\,$, one easily checks that 
\[
	D^{\g}\left(\left(r^{-k}\,\Pi\right)\star\a\right) 
	= 
	\left(r^{-k}\,\Pi\right)\star D^{\g}\a\ ,
\]
so that by our convolution formula, $\,\dar_{k}\rad_{k}\a\,$ is differentiable through order at least 
$\,q\,$ too.
\end{remark}

%% file: kernels2.tex

To fully exploit the convolution formula, we also need to know the Fourier transform of its kernel 
$\,r^{-k}\,\Pi\,$. We assume $\,k<n\,$, so the singularity at $\,r=0\,$ is integrable, and $\,r^{-k}\,\Pi\,$ 
represents a tempered distribution. It therefore has a Fourier transform. Explicitly, 

\begin{thm}\label{thm:fpi} When $\,0\le p\le k<n\,$, we have
\begin{equation*}
	\F\left(r^{-k}\Pi\right) 
	= \frac{1}{k}\,\frac{\sv{k}}{\sv{k'}}\
	\frac{(k-p)\,\Pi +(n-p)\,\Phi}{ r^{k'}}\ .
\end{equation*}
When $\,p<k\,$, this makes $\,\F\left(r^{-k}\Pi\right)\,$ pointwise invertible, with
\begin{equation*}
	\F\left(r^{-k}\Pi\right)^{-1}
	= \frac{\sv{k'}}{\sv{k}}\,k\,r^{k'}\,
	\left(\frac{\Pi}{k-p} + \frac{\Phi}{n-p}\right)\ .
\end{equation*}
\end{thm}

To prove this, we require two main ingredients. The first is Lemma \ref{lem:stein} below, 
a beautiful formula that Stein bases on an identity he credits to Hecke \cite[p.73]{stein}. 
The second, Lemma \ref{lem:shd}, prepares us to exploit Stein's lemma by expanding 
$\,\Pi\,$ as a sum involving spherical harmonics.

\begin{lem}\label{lem:stein}
Suppose $\,h_{d}:\R^{n}\to\C\,$ is a homogeneous harmonic polynomial
of degree $\,d\,$, and $\,0<k<n\,$. Then 
$\,r^{-d-k}h_{d}\,$ and its Fourier transform are both 
tempered distributions, with
\[
	\F\left(r^{-d-k}h_{d}\right)= 
	\i^{d}\ \frac{\sv{d+k}}{ \sv{d+k'}}\ r^{-d-k'}h_{d}
\]
\end{lem}

\begin{proof} Stein \cite[p.73]{stein}.
\end{proof}

\begin{remark}
	When $\,d=0\,$ and $\,h\equiv 1\,$, Lemma \ref{lem:stein} makes the well-known 
	classical assertion
	\[
		\sv{k'}\,\F\left(r^{-k}\right) = {\sv{k}}{}\, r^{-k'}\,.
	\]
	When $\,p=0\,$, Theorem \ref{thm:fpi} above reduces to precisely the same
	assertion, because the projections $\,\Pi\,$ and $\,\Phi\,$ reduce to the identity and zero
	operators respectively in that case, as noted in \S2. 
\end{remark}

Our next lemma handles the \emph{non}-triviality of $\,\Pi\,$ and $\,\Phi\,$ when $\,p>0\,$.
Recall that they are the operator fields on $\,\R^{n}\,$ which, at each non-zero 
$\,x\in\R^{n}\,$, project $\,\bw{*}{\R^{n}}\,$ to $\,\bw{*}H_{x}\,$ and $\,x\wedge\bw{*}H_{x}\,$ 
respectively (cf. \ref{eqn:piPhi}). The wedge operation and its adjoint (\ref{eqn:vee}) yield an 
elegant and useful pair of formulae for these operations:

\begin{obs}\label{obs:veewedge}
When $\,0\ne x\in\R^{n}\,$, and $\,\a\in\bw{*}\R^{n}\,$, we have
	\begin{eqnarray*}
		|x|^{2}\,\Pi_{x}(\alpha) &=& (x\wedge \a)\vee x\\
		|x|^{2}\,\Phi_{x}(\alpha) &=& x\wedge (\a\vee x)\ .
	\end{eqnarray*}
\end{obs}

\begin{proof}
	Extend $\,e_{1}:=x/|x|\,$ to an orthonormal basis $\,\{e_{1},\dots,e_{n}\}\,$ for $\,\R^{n}$. 
	It then suffices to check both formulae on the corresponding multivector basis 
	$\,e_{\lambda}\,$ defined by (\ref{eqn:basis}), an easy exercise.
\end{proof}

We call an operator field $\,H:\R^{n}\to\mathrm{Hom}(\bw{*}\R^{n})\,$
\textbf{harmonic} if
\[
	\Delta \langle H(\a),\b\rangle=0
\]
for all constant elements $\,\a,\b\in\bw{*}\R^{n}$. 

\begin{lem}\label{lem:shd}
Let $\,0\le p\le n\,$. Then there exists a 
{homogeneously quadratic} and {harmonic} operator field 
$\,H\,$ such that, as operators on $\,p$-forms, we have
	\begin{eqnarray*}
		\Pi &=& ({p'\,/n})\,I + r^{-2}\,H\\
		\Phi &=& ({p/n})\,I- r^{-2}\,H\ 
	\end{eqnarray*}
Here $\,I\,$ is the identity on $\,\bw{p}{\R^{n}}\,$ and $\,p'=n-p\,$.
\end{lem}

\begin{proof}
Since $\,\Pi+\Phi=I\,$, it suffices to prove the first identity
above, or equivalently, that
\[
	n\,r^{2}\,\Pi = p'\,r^{2}\,I + n\,H
\]
with $\,H\,$ harmonic. 

Observation \ref{obs:veewedge} says, 
for any $\,\a\in\bw{p}\R^{n}\,$, that
\[
	n\,r^{2}\,\Pi(\a) = n\,(x\wedge \a)\vee x\ ,
\]
showing that $\,n\,r^{2}\,\Pi\,$ is polynomially
quadratic. Since the ring of 
homogeneous quadratic polynomials on $\,\R^{n}\,$ splits
as $\,r^{2}\C + \mathcal{H}^{2}\,$, where $\,\mathcal{H}^{2}\,$
denotes the harmonic subspace, we can write
\begin{equation}\label{eqn:AH}
	n\,r^{2}\,\Pi = r^{2}A + n\,H\ ,
\end{equation}
with $\,A\,$ constant, and $\,H\in\mathcal{H}^{2}\,$. 
Apply the laplacian to this expansion and divide by $\,n\,$ to get
\begin{equation}\label{eqn:lap1}
	 \Lap(r^{2}\,\Pi)=2\,A\ .
\end{equation}

On the other hand, Observation \ref{obs:veewedge} gives,
for any basis $p$-vector of the form
$\,e_{\lambda}:= e_{\lambda_{1}}\wedge e_{\lambda_{2}}\wedge
\cdots\wedge e_{\lambda_{p}}\,$ (cf. \ref{eqn:basis})
\[
	r^{2}\,\Pi_{x}(e_{\lambda}) = 
	\sum_{i,j}x_{i}x_{j}\,(e_{i}\wedge e_{\lambda})\vee e_{j}\ .
\]
Since $\,\Lap(x_{i}x_{j})= 2\,\delta_{ij}\,$, and
\[
	(e_{i}\wedge e_{\lambda})\vee e_{i}=
	\begin{cases} e_{\lambda}, & i\not\in\lambda\\ 0, & i\in\lambda\ ,
	\end{cases}
\]
we see that
\begin{equation*}
	\Lap\left(r^{2}\,\Pi(e_{\lambda})\right)
	=2\sum_{i}(e_{i}\wedge e_{\lambda})\vee e_{i}
	=2\sum_{i\not\in\lambda}e_{\lambda}=2\,p'\,e_{\lambda}\ ,
\end{equation*}
Since the $\,e_{\lambda}$ form a basis for $\,\bw{p}\R^{n}\,$, this means $\,\Lap(r^{2}\,\Pi) = 2p'\,I\,$, 
and comparison with (\ref{eqn:lap1}) then yields $\,A=p'\,I\,$, turning (\ref{eqn:AH}) into the identity 
we sought.
\end{proof}


\begin{proof}[\textbf{\emph{Proof of Thm.~\ref{thm:fpi}}}]
To compute the Fourier transform of $\,r^{-k}\Pi\,$,
first expand using the lemma just proven, to get
\begin{equation*}
	r^{-k}\Pi = \frac{p'}{n}\,\frac{I}{ r^{k}} + \frac{H}{ r^{k+2}}\ .
\end{equation*}
Since $\,I\,$ and $\,H\,$  are harmonic, and homogeneous
of degree $\,0\,$ and $\,2\,$ respectively,  
Lemma \ref{lem:stein} computes the Fourier transform of this 
expansion as
\begin{equation}\label{eqn:fpi1}
	\F(r^{-k}\Pi) =
	\frac{p'}{ n}\,\frac{\sv{k}}{ \sv{k'}}\,\frac{I}{ r^{k'}} - 
	\frac{|\US^{k+1}|}{ |\US^{k'+1}|}\,\frac{H}{ r^{k'+2}}\ .
\end{equation}
To simplify this, recall that 
in terms of the classical Gamma function, one has 
\[
	|\US^{k+1}|
	=
	\frac{2\,\pi^{\frac{k}{2}+1}}
	{\Gamma\left(\textstyle\frac{k}{ 2}+1\right)}
\]
for any dimension $\,k\,$. But $\,\Gamma(s+1)=s\,\Gamma(s)$, implying the recursive identity
\[
	|\US^{k+1}|=
	\frac{2\pi}{k}\,\sv{k}\ ,
\]
which reduces (\ref{eqn:fpi1}) to 
\begin{eqnarray}\label{eqn:fpi2}
	\F(r^{-k}\Pi) 
	&=&	\frac{\sv{k}}{ \sv{k'}}\,\frac{1}{ r^{k'}}\,
	\left(\frac{p'}{ n} \,I- \frac{k'}{ k}\,\frac{H}{ r^{2}}\right)\, .
\end{eqnarray}
Now use the first identity in Lemma \ref{lem:shd} 
to rewrite $\,H/r^{2}\,$ in terms of $\,\Pi\,$ and $\,I\,$, to get
\begin{eqnarray*}
	\frac{p'}{ n} \,I- \frac{k'}{ k}\,\frac{H}{ r^{2}}	
	&=& \frac{p'}{ n}\,I -\frac{k'}{ k}\left(\Pi-\frac{p'}{ n}\,I\right)\\
	&=& \frac{p'}{ n}\left(1+\frac{k'}{ k}\right)\,I - \frac{k'}{ k}\,\Pi\\
	&=& \frac{p'}{ k}\,I - \frac{k'}{ k}\,\Pi\\
	&=& \frac{p'}{ k}\left(\Pi+\Phi\right) - \frac{k'}{ k}\,\Pi\\
	&=& \frac{(k-p)\,\Pi + (n-p)\,\Phi}{ k}\ ,
\end{eqnarray*}

since $\,p'-k'=k-p\,$. Inserting this into (\ref{eqn:fpi2}), 
we obtain the desired formula for $\,\F(r^{-k}\Pi)\,$. 

The subsequent formula for its inverse now follows
immediately when $\,p<k\,$, because $\,\Pi\,$ and $\,\Phi\,$ project onto orthogonal complements, so that $\,\Pi\Phi=\Phi\Pi=0\,$ 
and $\,\Pi+\Phi=I\,$.
\end{proof}

%% file: inversion2.tex

Finally, we will combine the Convolution formula (Prop.~\ref{prop:conForm}) with 
our computation of $\,\F(r^{-k}\Pi)\,$ (Thm.~\ref{thm:fpi}) to invert $\,\rad_{k}\,$
explicitly on the space of Schwartz $p$-forms.

When $\,p=0\,$, our formula reduces to the well-known inversion of the
classical $k'$-plane transform. 
The latter result applies well beyond the Schwartz class, to all continuous functions
 \cite[Thm.~6.2]{helgason} (or even functions in $\,L_{\mathrm{loc}}^{q}\,$, with $\,1\le q\le n/k'\,$ 
 \cite[Thm.~4.1]{rubin}), with $\,O(r^{-n-\eps})\,$ decay.  
 
We expect that, likewise, our inversion formula holds for $p$-forms with 
$\,O(r^{-n-\eps})\,$ decay, leading to both a superposition formula, and an algorithm for 
reconstructing currents from their projections. We will state the results we anticipate along
those lines. We leave their proofs, however, to colleagues more expert in the analytical techniques 
they seem to require.

\begin{thm}[Inversion]\label{thm:inversion}
	Suppose $\,k<n\,$, and $\,\a\,$ is a Schwartz $p$-form on $\,\R^{n}\,$, with
	$\,p<k\,$. Then we may recover $\,\a\,$ from its transform $\,\rad_{k}\a\,$ 
	as follows:
	\[
		\a = C_{n,k,p}\,\F^{-1}\left(r^{k'}\,
		\left(\frac{\Pi}{k-p} + \frac{\Phi}{n-p}\right)
		\dit\F\left(\dar_{k}\,\rad_{k}\a\right)\right)\ .
	\]
	Here $\,\F\,$ and $\,\F^{-1}\,$ apply in the
	tempered distributional sense, and
	\[
		C_{n,k,p}= k\, 
		\frac{\sv{n}\,\bn{n-1}{p}}{\sv{k}\,\bn{n-k}{p}}\ ,
	\]
\end{thm}

\begin{proof}
Apply the Convolution formula (Prop.~\ref{prop:conForm}) to get
\begin{equation*}\label{eqn:step1}
	\dar_{k}\,\rad_{k}\,\a = c_{1}
	\left(r^{-k}\Pi\right)\star\a\ ,
\end{equation*}
where 
\begin{equation*}\label{eqn:c1}
c_{1}=\frac{\sv{k'}\,\bn{k'}{p}}{ \sv{n}\,\bn{n-1}{p}}\ .
\end{equation*}
As noted at the start of \S\ref{sec:FT}, $\,r^{-k}\,\Pi\,$ represents a tempered  
$\,\mathrm{Hom}(\bw{*}\R^{n})$-valued distribution on $\,\R^{n}$. 
The multiplication rule for $\,\F\,$ therefore extends to the convolution above
(see Observation \ref{obs:conv} for a proof), and we can
write its Fourier transform as 
\begin{equation*}\label{eqn:inv1}
	\F\left(\dar_{k}\rad_{k}\a\right) 
	= c_{1}\ \F\left(r^{-k}\Pi\right)\dit\F(\a)\ .
\end{equation*}
Theorem \ref{thm:fpi} computes $\,\F(r^{-k}\Pi)\,$ explicitly 
as an operator field, and shows that the right-hand side above, though apriori just a 
tempered distribution, is actually represented by a locally integrable $p$-form. We can then
multiply both sides by the field pointwise inverse to that $p$-form 
(also given by Theorem \ref{thm:fpi}) to deduce
\begin{eqnarray*}\label{eqn:fa}
	\F(\a) 
	&=&
	\frac{c_{2}}{c_{1}}\,\,r^{k'}\,
	\left(\frac{\Pi}{k-p} + \frac{\Phi}{n-p}\right)
	\dit\F\left(\dar_{k}\,\rad_{k}\a\right)
\end{eqnarray*}
where 
\begin{equation}\label{eqn:c2}
	c_{2}=c_{2}(n,k):=k\,\frac{\sv{k'}}{\sv{k}}\ .
\end{equation}
The Fourier transform preserves Schwartz spaces, so $\,\F(\a)\,$, like $\,\a\,$ itself,
belongs to $\,\sz{p}{\R^{n}}\,$. In particular, the right side of the formula above lies
in $\,\sz{p}{\R^{n}}\,$, and we can therefore recover $\,\a\,$ from $\,\rad_{k}\a\,$ by applying 
the inverse Fourier transform:
\begin{equation}\label{eqn:case1}
	\a =
	\frac{c_{2}}{c_{1}}\,
	\F^{-1}\left(
	\,r^{k'}\,
	\left(\frac{\Pi}{k-p} + \frac{\Phi}{n-p}\right)
	\dit\F\left(\dar_{k}\,\rad_{k}\a\right)
	\right)\ .
\end{equation}
Since $\,c_{2}/c_{1}=k\,C_{n,k,p}\,$, this completes the proof.
\end{proof}

When the codimension $\,k'=n-k\,$ is even, the unwieldy formula above takes a simpler
``differential'' form, using the operators $\,d\,$ and $\,\delta\,$ instead of the Fourier transform:

\begin{cor}\label{cor:even}
When the hypotheses of Theorem \ref{thm:inversion} hold with codimension 
$\,k'=:2l\in2\mathbf{Z}\,$, we can recover $\,\a\,$ from $\,\rad_{k}\a\,$ via
\[
	\a 
	= 
	C_{n,k,p}
	\left(
		\frac{-1}{4\pi^{2}}\right)^{l}
	\left(
		\frac{(\delta d)^{l}}{k-p}+\frac{(d\delta)^{l}}{n-p}
	\right)
	\left(\dar_{k}\,\rad_{k}\a\right)\ .
\]
\end{cor}

\begin{proof}
By the formulae in Proposition \ref{prop:symbols2}, $\,\F\,$ intertwines
$\,d\delta\,$ and $\,\delta d\,$ with multiplication by $\,r^{2}\Phi\,$ and $\,r^{2}\Pi\,$ respectively. 
But $\,\Phi\,$ and $\,\Pi\,$ are projections, so $\,\Phi^{l}=\Phi\,$ and $\,\Pi^{l}=\Pi\,$ for any 
integer $\,l>0\,$, whereby
\begin{eqnarray*}
	\F^{-1}r^{2l}\Phi\,\F
	&=& 
	\left(\F^{-1}r^{2}\Phi\,\F\right)^{l}
	=
	\left(-\frac{d\delta}{4\pi^{2}}\right)^{l}\,,\\
	\F^{-1}r^{2l}\Pi\,\F
	&=& 
	\left(\F^{-1}r^{2}\Pi\,\F\right)^{l}
	=
	\left(-\frac{\delta d}{4\pi^{2}}\right)^{l}\,.
\end{eqnarray*}
When $\,k'=2l\,$, these facts simplify our 
inversion formula in an obvious way. Remark \ref{rem:smooth} then
ensures that $\,\dar_{k}\rad_{k}\a\,$ is suitably differentiable,
and the corollary follows. 
\end{proof}

\begin{remark}\label{rem:p=0}
	The inversion and superposition results above all correctly reproduce their classical 
	scalar precedents when $\,p=0\,$, since $\,\Pi\,$ is the identity and $\,\Phi=0\,$ on 
	$\,\bw{0}(\R^{n})\,$. Indeed, one easily checks that these facts reduce the inversion 
	formula of Theorem 	\ref{thm:inversion} to \medskip
	\[
		f = \frac{\sv{n}}{\sv{k}}\,\F^{-1}\left(r^{k'}\,\F\left(\dar_{k}\,\rad_{k}f\right)\right)\,,
	\]
	\vskip 2pt
	for any Schwartz function $\,f\,$. For functions, $\,\delta d=0\,$ and $\,d\delta\,$ is the Laplace
	operator $\,\Delta\,$. So when $\,k'=2l\,$ is even and $\,p=0\,$, Corollary \ref{cor:even} 
	reduces to\smallskip
	\[
		f 
		= 
		\left(\frac{-1}{4\pi^{2}}\right)^{l}
		\frac{\sv{n}}{\sv{k}}\,{\Delta^{l}}
		\left(\dar_{k}\,\rad_{k}f\right)\,,
	\]
	\vskip 4pt
	in agreement with \cite[Thm.~6.2]{helgason}. 
			
	Indeed, when $\,p=0\,$ the inversion formula extends beyond $\,\sz{p}{\R^{n}}\,$ 
	to (for instance) the class of all continuous functions $\,\a\,$ with $\,O(r^{-n-\eps})\,$ decay near infinity,
	as mentioned above, and we expect the analogous extension to hold for $\,p>0\,$:

\end{remark}

\begin{conj}\label{conj:extend}
	The inversion formulae of Theorem \ref{thm:inversion} and Corollary
	\ref{cor:even} hold for all continuous $p$-forms $\,\a\,$ with $\,O(r^{-n-\eps})\,$ decay near infinity 
	for some $\,\eps>0\,$.
\end{conj}

Now recall again that when $\,\a\in\sz{p}{\R^{n}}\,$, so is $\,\F\a\,$. This fact makes it easy to show that
\begin{equation}\label{eqn:phi}
	\phi:=
		\F^{-1}
			\left(r^{k'}
				\left(
					\frac{\Pi}{k-p} + \frac{\Phi}{n-p}
				\right)
				\dit
				\F\a
			\right)
\end{equation}
is smooth, with an $\,O(r^{-n})\,$ decay estimate. In particular, $\,\phi\,$ is integrable on
any $k'$-plane, so that
\[
	\beta:= C_{n,k,p}\,\rad_{k}\phi
\]
(with $\,C_{n,k,p}\,$ as in Theorem \ref{thm:inversion}) is well-defined and smooth on $\,\can{n,k}\,$. 
A slightly better $\,O(r^{-n-\eps})\,$ decay bound for $\,\phi\,$ would combine with our Conjecture above
and our inversion formula to let us write
\begin{equation}\label{eqn:phi2}
	\phi = 
		C_{n,k,p}\,\F^{-1}
			\left(
				r^{k'}\left(
					\frac{\Pi}{k-p}+\frac{\Phi}{n-p}
				\right)
				\dit
				\F\left(
					\dar_{k}\rad_{k}\phi
				\right)
			\right).
\end{equation}
But the operator
\[
	\F^{-1}\left(
		r^{k'}\left(
			\frac{\Pi}{k-p}+\frac{\Phi}{n-p}
		\right)\right)
		\dit
		\F
\]
is clearly invertible, and hitting both sides of (\ref{eqn:phi2}) with its inverse, given the definition   
of $\,\phi\,$ in (\ref{eqn:phi}), would then give $\,\a=\dar_{k}\b\,$, making $\,\a\,$ a superposition 
of $k$-planar $p$-forms. So if Conjecture \ref{conj:extend} holds, so does

\begin{conj}[Superposition]\label{cor:super}
	Given $\,0\le p<k<n\,$, any $\,\a\in\sz{p}{\R^{n}}\,$ can be constructed as a
	superposition of $k$-planar $p$-forms by a formula of the type
	\[
		\a=C_{n,k,p}\dar_{k}\rad_{k}\phi\,,
	\]	
	where $\,\phi\,$ is given by (\ref{eqn:phi}) above.

\end{conj}

As we have noted, this has been proven when $\,p=0\,$ (\cite[Thm.~6.2]{helgason}). 
In view of Remark \ref{rem:superpose}, it then analyzes the scalar-valued Schwartz 
function $\,\a\,$ as a superposition of ``$k'$-plane waves,'' i.e., 
functions invariant with respect to the translations generated by $\,k'\,$ independent vectors.
That fact accounts for the classical use of Radon transform methods to solve certain linear partial 
differential equations by the ``method of spherical means.'' (cf. \cite[\S7a]{helgason} and \cite[(1.1)]{rubin}).

Conjecture \ref{cor:super} would similarly let us analyze Schwarz $p$-forms as superpositions 
of $k$-planar forms. This in turn would allow us to reconstruct $p$-\emph{currents} from 
their projections. 

To explain this, we introduce the linear space $\,\E{p}{\R^{n}}\,$ comprising all smooth $p$-forms 
on $\,{\R^{n}}\,$. One topologizes $\,\E{p}{\R^{n}}\,$ in a standard way, so that a sequence in 
$\,\E{p}{\R^{n}}\,$ converges to zero if and only if its members, together with all their derivatives, 
converge uniformly to zero on any fixed compact subset of $\,\R^{n}\,$ (cf. \cite[Chap. V]{helgason}).

\begin{definition}[Currents]
	We call any continuous linear functional on $\,\E{p}{\R^{n}}\,$
	a \textbf{compactly supported $p$-current on $\,\R^{n}\,$}. 
\end{definition} 

One easily shows that any such current actually does have
compact support, in the sense that for some $\,R>0\,$, 
it vanishes on all smooth $p$-forms supported outside 
the ball $\,B_{R}(0)\,$. Since we can find a Schwarz $p$-form 
that agrees, on $\,B_{R}(0)\,$ with any given smooth $p$-form, 
it follows immediately that
 
\begin{obs}\label{obs:restrict}
	A compactly supported $p$-current 	is determined by its values on 
	$\,\sz{p}{\R^{n}}\subset\E{p}{\R^{n}}\,$.
\end{obs}

Finally, note that we can map any such current $\,T\,$ into a $k$-plane 
$\,P\in\G{n,k}\,$, to produce a compactly supported 
$k$-current  called the \textbf{projection of $\,T\,$
into $\,P\,$}, and denoted by $\,P_{*}T$, via the simple rule
\begin{equation*}\label{eqn:projT}
	P_{*}T[\phi]:= T[P^{*}\phi]\quad
	\text{for all $\,\phi\in\E{p}{P}\,$}\,.
\end{equation*}
Here, as in our introduction, $\,P^{*}\phi\,$ denotes the 
pullback of $\,\phi\,$ from $\,P\,$ to $\,\R^{n}\,$ via the
orthogonal projection $\,\R^{n}\to P\,$. 

\begin{conj}\label{thm:proj}
	When $\,p<k< n\,$, each compactly supported $p$-current $\,T\,$ on 
	$\,\R^{n}\,$ is uniquely determined by its projections into
	all $k$-planes $\,P\in\G{n,k}\,$. Explicitly, for any 
	$\,\a\in\sz{p}{\R^{n}}\,$, we have
	\[
		T[\a]
		=
		\int_{G{n,k}}P_{*}T\left[\b(P,\,\cdot\,)\right]\ dP\ ,
	\]
	where $\,\b\,$ is the mapping
	constructed from $\,\a\,$ in Theorem \ref{cor:super}.
\end{conj}

\begin{proof}[Proof (modulo Conjecture \ref{conj:extend})]
	By Observation \ref{obs:restrict}, 
	it suffices to show that whenever $\,\a\in\sz{p}{\R^{n}}\,$, we can determine $\,T[\a]\,$
	from the projections $\,P_{*}T\,$ for every $\,P\in\G{n,k}\,$. But if Conjecture \ref{conj:extend}
	holds, then so does Conjecture \ref{cor:super}, and the latter constructs a $p$-form $\,\phi\,$ from 
	$\,\a\,$ for which the mapping $\,\b:=\rad_{k}\phi\,$ satisfies
	\[
		\a(x) 
		= 
		(\dar_{k}\b)(x) 
		= 
		\int_{\G{n,k}}\b(P,Px)\big|_{P}\ dP
	\]
	for every $\,x\in\R^{n}\,$. Since $\,\b\,$ lies in the image of $\,\rad_{k}\,$, 
	Remark \ref{rem:subbdl} makes the restriction $\,|_{P}\,$ superfluous in the formula 
	above, which we can thus rewrite as
	\[
		\a = 
		\int_{\G{n,k}}P^{*}\left(\b(P,\,\cdot\,)\right)\ dP\ .
	\]
	It would then follow that
	\begin{eqnarray*}
		T[\a]
		&=&
		T\left[\int_{\G{n,k}}P^{*}\left(\b(P,\,\cdot\,)\right)
		\ dP\right]\\
		&=&
		\int_{\G{n,k}}T\left[P^{*}\left(\b(P,\,\cdot\,)
		\right)\right]
		\ dP\\
		&=&
		\int_{\G{n,k}}P_{*}T\left[\b(P,\,\cdot\,)\right]\ dP\ ,
	\end{eqnarray*}
	proving the result.
\end{proof}

%% file: tech2.tex

To clarify the main logic of our exposition, we have deferred two technical facts about
convolution to this appendix.

%
%

We start by verifying that the notion of convolution extends
to the setting required by our inversion theorem.

\begin{obs}\label{obs:defcon}
When $\,F\in\sz{}{\R^n,V}\,$, 
the classical definition of $\,T\star F\,$ 
(see Item (\ref{itm:ftcnv}) in \S\ref{ssec:ft})
extends in a natural way to the case where $\,T\,$ is a 
tempered $(V,W)$-distribution. 
\end{obs}

\begin{proof}
To do this, we identify vectors with operators in a somewhat 
unusual way.

Let $\,v^{*}\,$ denote the image of $\,v\in V\,$ 
under the canonical duality $\,V\sim V^{*}\,$.
We then let $\,v_{\otimes}:W\to\mathrm{Hom}(V,W)\,$ signify 
the operator we obtain by setting
\begin{equation*}
	\left(v_{\otimes}(w)\right)(x) := (w\otimes v^{*})(x) 
	:= v^{*}(x)\,w\in W
\end{equation*}
for all $\,x\in V\,$ and $\,w\in W\,$. 
We can then identify any $V$-field $\,F\,$
on $\,\R^{n}\,$ with the field $\,F_{\otimes}\,$ given by
\begin{equation*}\label{eqn:ftens}
	x\mapsto F_{\otimes}(x):= F(x)_{\otimes}\ .
\end{equation*}
A routine calculation shows that whenever 
$\,T\in\sz{}{\R^{n},V,W}\,$, $\,F\in\sz{}{\R^{n},V}\,$ and
$\,G\in\sz{}{\R^{n}, W}\,$, we have 
\begin{equation}\label{eqn:ot}
	(T\dit F)[G] = T[G\otimes F^{*}]\ .
\end{equation}
Using this fact with a bit more calculation then reveals that
\begin{equation}\label{eqn:cnv1}
(T\star F)[G]=T[F_{\otimes}^{-}\star G]\ ,
\end{equation}
where the superscript ``$-$'' indicates reflection through the origin, e.g.,
\begin{equation*}\label{eqn:x-x}
 F^{-}(x):=F(-x)\ .
\end{equation*}

Since $\,F_{\otimes}^{-}\star G\,$ lies in 
$\,\sz{}{\R^n,V,W}\,$
(\ref{eqn:cnv1}) now {defines} $\,T\star F\,$ for any 
$\,T\in\szd{}{\R^n,V,W}\,$, as promised.
\end{proof}

We next show that the Fourier convolution rule 
(Item (\ref{itm:ftcnv}) of \S\ref{ssec:ft} again)
extends to convolutions of the type just described:

\begin{obs}\label{obs:conv}
The Fourier convolution/product rule 
\[
\F\left(T\star F\right) = \F(T)\dit\F(F)
\]
holds when $\,T\in\szd{}{\R^{n},V,W}\,$ and $\,F\in\sz{}{\R^{n},V}\,$.
\end{obs}

\begin{proof}
In view of (\ref{eqn:ot}) and (\ref{eqn:ft*}), 
we have to show that for any $\,G\in\sz{}{\R^{n},W}\,$, 
\begin{equation}\label{eqn:mk1}
	\left(T\star F\right)[\F(G)] 
	= T\left[\F\bigl(G\otimes\F(F)^{*})\right]\ .
\end{equation}
We start by using (\ref{eqn:cnv1}) to write
\[
\left(T\star F\right)\left[\F(G)\right]
=T\left[F_{\otimes}^{-}\star\F(G)\right]\ .
\]
Since $\,F\,$ and $\,G\,$ are both Schwartz, properties (\ref{itm:ftsz}), (\ref{itm:ftiso}) and (especially) 
the convolution/product rule
(\ref{itm:ftcnv}) listed above, together show that
\begin{eqnarray*}
F_{\otimes}^{-}\star\F(G) 
&=& \F^{-1}\left(\F(F^{-}_{\otimes})\dit \F^{2}(G)\right)\\
&=& \F\left(\F(F)^{-}_{\otimes}\dit G^{-}\right)^{-}\\
&=& \F\left(G\otimes\F(F)^{*})\right)\ .
\end{eqnarray*}
The desired fact (\ref{eqn:mk1}) follows immediately.
\end{proof}